\documentclass[12pt]{amsart}

\usepackage{amsmath,amssymb,amsbsy,amsfonts,latexsym,amsopn,amstext,cite,
amsxtra,euscript,amscd,bm,mathabx, mathrsfs} 
\usepackage{url}

\usepackage[colorlinks,linkcolor=blue,anchorcolor=blue,citecolor=blue,backref=page]{hyperref}
\usepackage{color}
\usepackage{graphics,epsfig}
\usepackage{graphicx}
\usepackage{float}
\usepackage{epstopdf}
\hypersetup{breaklinks=true}

\usepackage[np]{numprint}
\npdecimalsign{\ensuremath{.}}

\usepackage{float}
\usepackage[norefs,nocites]{refcheck}

\usepackage{mathtools}
\usepackage{todonotes}

\usepackage[T1,T2A]{fontenc}

\def\MultEll{$\textit{MultEll}$}
\def\tMultEll{$\widetilde{\textit{MultEll}}$}

\newtheorem{theorem}{Theorem}
\newtheorem{lemma}[theorem]{Lemma}

\newtheorem{cor}[theorem]{Corollary}

\newtheorem{definition}[theorem]{Definition}

\newtheorem{algor}[theorem]{Algorithm}

\numberwithin{equation}{section}
\numberwithin{theorem}{section}
\numberwithin{table}{section}
\numberwithin{figure}{section}

\newtheorem{rem}[theorem]{Remark}

\newcommand{\NN}{\mathbb{N}}

\newcommand{\ZZ}{\mathbb{Z}}

\renewcommand{\mod}{\,\mathop{\rm mod}}

\newcommand{\Z}{\mathbb{Z}}
\newcommand{\F}{\mathbb{F}}
\newcommand{\Q}{\mathbb{Q}}

\def\gcd{\operatorname{gcd}}
\def\vp{\operatorname{\psi}}

\def\mand{\qquad \text{and} \qquad}

\def\cE{{\mathcal E}}

\def\cI{{\mathcal I}}

\def\cN{{\mathcal N}}

\def\cP{{\mathcal P}}

\def\cR{{\mathcal R}}

\def\0{{\mathbf{0}}}

\def\({\left(}
\def\){\right)}
\def\fl#1{\left\lfloor#1\right\rfloor}

\begin{document}

\title{On oracle factoring of integers} 

\author[A. D\k{a}browski]{Andrzej D\k{a}browski}

\address{
Institute of Mathematics, University of Szczecin, Wielkopolska 15, 
70-451 Szczecin, Poland}
\email{andrzej.dabrowski@usz.edu.pl}
\email{dabrowskiandrzej7@gmail.com} 

\author[J. Pomyka\l a]{Jacek Pomyka\l a}

\address{Institute of Mathematics, Warsaw University, Banacha 2, 02--097 Warsaw, Poland}
\email{pomykala@mimuw.edu.pl}

\author[I. E. Shparlinski] {Igor E. Shparlinski}

\address{Department of Pure Mathematics, University of New South Wales,
Sydney, NSW 2052, Australia}
\email{igor.shparlinski@unsw.edu.au}

\date{}

\begin{abstract}
We present an oracle factorisation algorithm, which in polynomial deterministic time, finds a nontrivial factor of almost all 
positive integers $n$ based on the knowledge of the number of points on
certain elliptic curves in residue rings modulo $n$.
\end{abstract}


\maketitle{}

\section{Introduction}

\subsection{Motivation and background} 
It is widely believed that the integer factorisation problem, that is,  the problem of finding a non-trivial factor of 
a positive integer $n$ is a difficult computational question see~\cite[Section~5.5]{CrPom}. 
The best known rigorously proved deterministic 
algorithm is due to Harvey~\cite{Ha} with a slight improvement in~\cite{HaHi}, and runs in time $n^{1/5+o(1)}$, which builds on the $n^{2/9+o(1)}$-algorithm
of   Hittmeir~\cite {Hitt}, which, in turn, improves the previous $n^{1/4+o(1)}$-algorithm of Pollard~\cite{Pol}.
In fact, prior~\cite{Ha},  an $n^{1/5+o(1)}$-algorithm has also been known, however only conditional on  the Generalised Riemann 
Hypothesis (GRH), see~\cite[Section~5.5]{CrPom}.  
There are also fast probabilistic algorithms, some of which, such as the number field sieve,
remain heuristic, we refer to~\cite[Chapter~6]{CrPom} for an outline of these algorithms,
see also~\cite{LeeVen} for some recent progress towards a rigorous version of the number field sieve. 

It is also known from the work of Miller~\cite{Mil} that under the assumption of the GRH, computing the Euler function $\varphi(n)$ is deterministic polynomial time equivalent to 
computing a non-trivial factor of an integer $n$ or prove that it is prime. 

Shoup~\cite[Section~10.4]{Shoup} has given an unconditional version of this reduction, however his 
algorithm is probabilistic. 

These results of Miller~\cite{Mil} and Shoup~\cite[Section~10.4]{Shoup} have a natural interpretation
of {\it oracle factoring\/}. Namely, given an oracle, which for every integer $n\ge 1$ outputs $\varphi(n)$ 
we can factor $n$ in polynomial time. We also recall, several other oracle factoring algorithms, such as 
\begin{itemize}
\item a heuristic algorithm of Maurer~\cite{Maur} requiring certain $\varepsilon \log n$ oracle calls, which is based on the {\it elliptic curve factoring\/} algorithm
of Lenstra~\cite{Len2,Len1},
\item a rigorous algorithm of Coppersmith~\cite{Copp1,Copp2} requiring certain about $ 0.25 \log n$ oracle calls, which is based on an algorithm to find small solutions to polynomial congruences,
\item a rigorous algorithm of  Pomyka\l a and Radziejewski~\cite{PR}, with the oracle being the multiple $D$ of a value of the  Euler totient function $\varphi(n)$ with $D\le \exp\((\log n)^{O(1)}\)$, 
\end{itemize}
see also~\cite{LPZHL,LZL,MayRit,NA} and references therein. 

Sutherland~\cite[Chapter~2.3]{Suth} has designed a probabilistic factoring algorithm which uses an oracle that returns a multiple of  the multiplicative order of integers modulo $N$. 
Bach~\cite{Bach} established a similar result for a discrete logarithm oracle.

Here, we design an elliptic curve version of the results of Miller~\cite{Mil} and Shoup~\cite[Section~10.4]{Shoup}; 
we refer to~\cite{kn,Silv} for an appropriate background on elliptic curves. The  probabilistic approach for the oracle answering the order of points of elliptic curves modulo $n$  was investigated in \cite{MMV} with the related error probability $184/225$ (improved to $53/80$  in \cite{DP}).

We remark that the Euler function gives the number of elements of the unit group 
$\Z_n^*$ of the residue ring $\Z_n$ modulo $n$. Our elliptic curve analogue of this oracle 
returns the number of points on certain elliptic curves reduced modulo $n$. 
Furthermore, as in the case of the algorithm of Shoup~\cite[Section~10.4]{Shoup}, 
we only request a 
small multiple of these numbers rather than their exact values, see Definitions~\ref{def:MultEll} 
and~\ref{def:MultEll-t} below. 
Recently a related scenario (for a slightly 
different family of elliptic curves) has also been 
considered by Dieulefait and Urroz~\cite{DiUr}. However, here we do not assume that our oracle
has an access to integers $b$ with 
$$
\(\frac{b}{n}\) = -1
$$ 
for the Jacobi symbol. Instead we estimate the proportion of integers $n \le N$ for which this smallest value of $b$ is large.

\subsection{Notation and conventions} \label{sec:not}
Throughout the paper, the 
notations 
$$U = O(V), \qquad V = \Omega(U), \qquad U \ll V, \qquad V \gg U
$$ 
are all equivalent to the
statement that the inequality $|U| \le c V$ holds with some 
constant $c> 0$, which may occasionally, where obvious depend on
the real parameters $\gamma$, $\varepsilon$ and $\delta$.

As we have mentioned, for a positive integer $n$, we use $\Z_n$ to denote the residue ring modulo $n$ and we use 
$\Z_n^*$ to denote the group of units of $\Z_n$. 
For a prime $p$ we denote by $\F_p$ the finite field of $p$ elements.

For a square-free $n$ we use $N_n$ to denote the least positive integer 
$b$ with $\gcd(b,n) = 1$ and such that for the Jacobi symbol we have 
$$
\(\frac{b}{n}\) = -1.
$$ 
The oracles have to treat in fact only the square-free values of $b\le B$, since 
obviously $N_n$ is a  prime number. 

The letters $p$ and $q$ with or without indices, always denote prime numbers. 

We measure the time of our algorithms in the number of bit operations. However we also use
naive arithmetic algorithms and the recent striking progress of  D. Harvey and  
J. van der Hoeven~\cite {HavdH}  does not affect our finally results (as the possible 
advantage gets absorbed in $o(1)$ in the powers of $\log N$).

\subsection{Background on elliptic curves in residue rings}

For an elliptic curve $E$ over $\mathbb Q$ given by a minimal 
Weierstrass equation with integer coefficients see~\cite[Chapter~III, Section~1]{Silv}, 
we denote by 
$E(\mathbb Z_n)$ the set of solutions of the corresponding cubic congruence modulo $n$. 

Let us recall basic facts on elliptic curves defined over $\Z_n $ (see~\cite{Len2, Len1})
which is needed in this paper.
Assume that 
$$
n=\prod_{i=1}^s p_i
$$
is  a square-free integer with primes $p_i>3$. 
The projective plane $\mathbb P^2(\Z_n)$ is defined to be the set of equivalence classes of primitive triples in $\Z_n^3$ (that is, triples $(x_1,x_2,x_3)$ with $\gcd(x_1,x_2,x_3,n)=1$) with respect to the equivalence $(x_1, x_2,x_3)\sim(y_1, y_2,y_3)$ if $(x_1, x_2,x_3) = u(y_1, y_2,y_3)$ for a unit $u\in \Z^*_n$. An elliptic curve over $\Z_n$ is given by the short Weierstrass equation $E: y^2z = x^3 + axz^2 +bz^3$, where $a,b\in \Z_n$ and
the discriminant $-16(4a^3 + 27b^2)\in \Z_n^*$. 
Let $E(\F_{p_i})$ be the group of $\F_{p_i}$-rational points on 
the reduction of $E$ modulo $p_i$ for primes $p_i\mid n$. For the set $E(\Z_n)$ of points in $\mathbb P^2(\Z_n)$ satisfying the equation of $E$, by the Chinese remainder theorem
there exists a bijection 
\begin{equation}\label{phi1} \vp: E(\Z_n) \to E(\F_{p_1}) \times \ldots \times E(\F_{p_s})
\end{equation} 
induced by the reductions modulo $p_i$, $i =1, \ldots, s$. 

The points $(x:y:z)\in E(\Z_n)$ with $z\in \Z_n^*$ can be written as $(x/z: y/z:1)$ and are called finite points. The set $E(\Z_n)$ is a group with the 
addition for which $\vp$ is a group isomorphism, which in general can be defined using the so-called complete set of addition laws on $E$ 
(see \cite{Len2}). To add two finite points $P,Q\in E(\Z_n)$ with
$$
\vp(P) = (P_1,\ldots,P_s),\, \vp(Q) = (Q_1,\ldots,Q_s)\in E(\F_{p_1}) \times \ldots \times E(\F_{p_s})
$$
we can also use the same 
formulas as for elliptic curves over fields in the following two cases: 
\begin{itemize}
\item either $Q_i\neq \pm P_i$ for each $i=1, \ldots, s$;
\item or $Q_i = P_i$ and $Q_i\neq -P_i$ for each $i=1, \ldots, s$. 
\end{itemize}
In these cases, 
\begin{equation}\label{addel}
\begin{cases} x_{P+Q} = \lambda^2 - x_P - x_Q,\\
y_{P+Q} = \lambda(x_P - x_{P+Q}) - y_P,
\end{cases}
\end{equation} 
where 
$$
\lambda = \begin{cases}\displaystyle{ \frac{y_Q- y_P}{x_Q-x_P} }& \text{ if } Q_i \neq \pm P_i 
\text{ for } i=1, \ldots, s, \\ 
\displaystyle{ \frac{3x_P^2 +a}{2y_P} }& \text{ if } Q_i = P_i \text{ and } Q_i \neq -P_i
\text{ for } i=1, \ldots, s. 
\end{cases}
$$

In the remaining cases, these formulae will give not correct 
answer for the sum P+Q, because in these cases denominators 
in $\lambda$ will be zero divisors in $\mathbb Z_n$ (see~\cite[Section~3]{Len2}, for further discussion how to remedy this situation). 

Given a positive, square-free integer $n$, an odd prime divisor $r\mid n$ and 
$P=(a,b) \in E( \ZZ_n)$
we define its reduction $P \mod r \in E( \ZZ_r)$ modulo $r$ by 
$P \mod r = (a \mod r, b \mod r)$.

\subsection{Main results}
First we describe a \MultEll $(N,B,M)$-oracle which we assume is available to us.
Let 
$$
E_b:~y^2=x(x^2 - b)
$$ be an elliptic curve over $\Q$.

By the Chinese remainder theorem we see that for  a square-free $n$, the number of solutions $E(n,b)$ to the congruence 
$$
y^2\equiv x(x^2 - b)\mod n, \qquad (x,y) \in \Z_n^2,
$$
is given by  
\begin{equation}\label{eq:E(n,b)}
E(n,b)=\prod_{p\mid n}\#\(E_b(\F_p)\), 
\end{equation} 
provided that $\gcd(b,n)=1$, 
see also~\eqref{phi1}.

We recall that $N_n$ is prime, thus  the oracles below need to treat only 
square-free numbers. 

\begin{definition}[\MultEll$(N,B,M)$-oracle] 
\label{def:MultEll}
Given the parameters $B$,  $N$ and $M$,  for each  integer  $n \le N$ it returns
\begin{itemize}
\item   a positive multiple  
$$k_b E(n,b) \le M, \qquad k_b\in \NN,
$$ 
of $E(n,b)$ given by~\eqref{eq:E(n,b)}
for every $b\le B$ with $\gcd(b,n)=1$, if $n$ is square-free;
\item
an error message $\star\star\star$, if $n$ is not  square-free. 
\end{itemize}
\end{definition}

We are now ready to formulate our main result.

\begin{theorem}\label{thm:factoring} 
Let 
$$
2< \gamma <2+ \frac{\sqrt {257} -15}{16} \mand \delta<1/4
$$ 
be fixed. Assume that for a sufficiently large integer $N$ we are given a \MultEll $(N,B,M)$-oracle
where 
$$
B = \fl{(\log N)^\gamma} \mand M=N^{O(1)} . 
$$
Then there is a deterministic algorithm that finds a nontrivial factor of all   integers 
$n\le N$ with at most $N/(\log N)^{\delta}$
exceptions, in polynomial, deterministic time $O\((\log N)^{\rho(\gamma) +o(1)}\)$, 
where 
$$
\rho(\gamma) = \max\left\{\gamma +2, \frac{9\gamma-17}{8(\gamma-2)}\right\}
$$
and $O \left((\log N)^{\gamma }\right)$ oracle queries.
\end{theorem}

We note that the proof of Theorem~\ref{thm:factoring} is constructive and we actually exhibit a relevant algorithm in Section~\ref{sec:Alg},
see Algorithm~\ref{alg:AlgA}.
Clearly the limit of the method  in term of the number of queries is attained for $\gamma \to 2$. 
On the other hand, assuming that the queriescost only the reading time of the returns, 
that is, $O(\log N)$, one can choose $\gamma$ from the equation 
$$
\gamma +2 = \frac{9\gamma-17}{8(\gamma-2)}
$$
to minimise $\rho(\gamma)$, that is, 
$$
\gamma = \frac{9+\sqrt{561}}{16} = 2.0428 \ldots  \mand  \rho(\gamma) = 4.0428 \ldots \,.
$$ 

We can reduce the value of $\gamma>2 $ to $\gamma=1$ at the cost of a 
somewhat weaker estimate for exceptional numbers. 
In what follows we use a slightly modified \tMultEll$(N,B,M)$-oracle.

\begin{definition}[\tMultEll$(N,B,M)$-oracle] 
\label{def:MultEll-t}
Given the parameters $B$,  $N$ and $M$, for each $n \le N$, it 
returns  
\begin{itemize}
\item a positive multiple
$$k F(n,B) \le M, \qquad k\in \NN,
$$
of 
$$
F(n,B)=\prod_{b \le B} E(n,b),
$$ 
 if $n$ is square-free;
\item
an error message $\star\star\star$, if $n$ is not  square-free. 
\end{itemize}
\end{definition}

\begin{theorem}\label{thm:factoring2} 
There are some positive constants $c$ and $C$ such that if 
for a sufficiently large integer $N$ we are given a \tMultEll$(N,B,M)$-oracle, 
where 
$$
B = \fl{c \log N } \mand M=N^{O(B)}, 
$$
then there is a deterministic algorithm that finds a nontrivial factor of all  integers 
$n\le N$ with at most $ N\(\log N\)^{-C/\log\log\log N}$
exceptions, in polynomial, deterministic time $\log^{3+o(1)} N$,
with only one query of the \tMultEll$(N,B,M)$-oracle.  
\end{theorem}

\begin{rem} We note that we assume  that our oracles output results only for squarefree 
integers, which implicitly requires testing for squarefreeness (which in general is considered to be as hard as the integer factorisation problem). However, in our scenario
we can easily supplement our algorithms with a trial division testing for divisibility by 
integer squares $d^2$ with $d \le \log N$ and notice that there are only 
$O(N/\log N)$ non-square-free integers $n \le N$ that  pass this preselection, which 
is smaller than the sizes of exceptional sets in Theorems~\ref{thm:factoring} 
and~\ref{thm:factoring2}.  
\end{rem}

\subsection{Computational complexity of oracles   \MultEll\  and 
\tMultEll}  We note that the oracles make sense only for $M \gg N$. We also always have 
$$
B = (\log N)^{O(1)} \mand M \le N^{O(B)}.
$$ 
We use these inequalities when we  simplify
our estimates below.
 
Let us fix $b\le B$. We can find the factorisation of $n\le N$  in $N^{1/5+o(1)}$ deterministic time~\cite{Ha}.
Next we compute $E(p,b)$ for prime $p\mid n$ using the {\it Schoof algorithm\/} in $(\log N) ^{5+o(1)}$ deterministic time (see~\cite{BSS}).
Let $\omega(n)$ be the number of distinct prime factors of $n$.
We compute the product of $E(p,b)$ over primes $p\mid n$ in 
$$ \omega(n) (\log N )^{1+o(1)} \le  (\log N)^{2+o(1)}
$$  
deterministic time (see the end of Section~\ref{sec:concl}).
Similarly, we compute
a product $k_b E(n,b)\le M$, 
in deterministic time $(\log M)^{1+o(1)}$.
Hence
the total complexity is 
$$B(\log N)^{2+o(1)}+B(\log M)^{1+o(1)}+N^{1/5+o(1)}=N^{1/5+o(1)}.
$$
In case of  \tMultEll\ oracle we have additionally computation of $F(n,B)$ requiring  at most $O(B)$ multiplication of numbers of size $O\(B\log N\)$ giving the deterministic time $B (B\log N)^{1+o(1)}$.
So the total complexity of this oracle  is 
$$N^{1/5+o(1)}+B(B\log N)^{1+o(1)}= N^{1/5+o(1)}.
$$

\section{Preliminary results}

\subsection{Counting some special integers} 
\label{sec:sp int}

We need the following bound on the number of positive integers $n \in [x-z,x]$ which are 
free of prime divisors from some dense set of primes, see~\cite[Corollary~2.3.1]{hr}.

\begin{lemma} \label{lem:hr} 
Let  $x \ge z > 1$  and let $\cP$ be a set of primes such that for some positive constants $\delta$ and $A$ we have 
$$
\sum_{\substack{p < z \\p\in \cP}} \frac{1}{p} \ge \delta \log\log z -A. 
$$
Then we have
$$
\# \{n:~ x-z<n\le x, \ p\mid n \Longrightarrow p\notin \cP \} \ll \frac{z}{(\log z)^\delta}. 
$$ 
\end{lemma}

Covering  the interval $[1,N]$ by $O\(N/z\)$ intervals of the form $[x-z,x]$
we see that  Lemma~\ref{lem:hr}  implies the following. 

\begin{cor} \label{cor:hr} 
Let  $x \ge z > 1$  and let  $\cP$ be a set of primes such that for some positive constants $\delta$ and $A$ we have 
$$
\sum_{\substack{p < z \\p\in \cP}} \frac{1}{p} \ge \delta \log\log z-A. 
$$
Then we have
$$
\# \{n:~ 1 \le n\le x, \ p\mid n \Longrightarrow p\notin \cP \} \ll \frac{x}{(\log z)^\delta}. 
$$ 
\end{cor}

We denote by $\cN$ the set 
of all odd, positive, square-free integers and by $\cN(N)$ its restriction to the interval $[1,N]$. 
By $\cN(N,z)$ we denote the subset of $\cN(N)$ which consist of integers free of prime divisors $p \le z$.

\begin{lemma} \label{12345}
Let $z\in [\log N, \exp\((\log N)^\delta\) ]$ and let $\delta<1/4$ be an arbitrary positive constant. Then we have 
\begin{align*} 
\#\{n\in {\cN }(N,z):~ n=pqm, & \ p\neq q~\text{primes}, \ p,q \equiv 3 \mod 8, \ m \in \Z \}\\
& = \# {\cN }(N,z) + O\(N(\log N)^{-\delta} \log \log N\).
\end{align*} 
\end{lemma}

\begin{proof} 
Let $\#\cN_1(N,z)$ stand for the number of integers of $\cN(N,z)$ that are divisible by some prime $p\equiv 3 \mod 8$ and let $\#\cN_2(N,z)$ be the number of  $n \in \cN(N,z)$ that are divisible by at least two distinct primes $p,q\equiv 3 \mod 8$.

By the Prime Number Theorem in arithmetic progressions (see also~\cite{LaZa} and references therein), we have 
$$
\sum_{\substack{p < z \\p\equiv 3 \mod 8}} \frac{1}{p} =  \frac{1}{\varphi(8)}  \log\log z + O(1)
 =  \frac{1}{4} \log\log z + O(1).
$$ 
Now we see that  Corollary~\ref{cor:hr}  implies that 
 $$
 \#\cN_1(N,z)=\#\cN(N,z) +O\left((N/(\log N)^\delta\right) 
$$
 and therefore  
\begin{equation} \label{eq: N2 Asymp}
\begin{split} 
\#\cN_2(N,z)& =\#\cN_1(N,z) + O\left(E\right)\\
&=\#\cN(N,z) + O\left(E + N/(\log N)^\delta\right),
\end{split} 
\end{equation}  
where $E$ stands for an upper bound for the set  of the square-free numbers of type $n=pm\le N$, where $m$ is free of prime factors $q \equiv 3 \mod 8$.
Thus separating the cases $m=1$ and $m \neq1$ we obtain
\begin{equation} \label{eq: 3_8-free} 
\begin{split} 
E &  \ll N/(\log N)+
\sum_{z\le p\le N/z}\ \sideset{}{^*}\sum_{m \le N/p} 1 \\
& \ll N/(\log N)+\sum_{z<p\le N/z} \frac{N/p}{(\log (N/p))^\delta}, 
\end{split} 
\end{equation}
where $\Sigma^*$ means that the summation is over positive integers $m>1$ free of prime factors $q \equiv 3 \mod 8$.
Therefore it remains to prove that the last sum satisfies 
\begin{equation} \label{eq: goal} 
\sum_{z<p\le N/z} \frac{N/p}{(\log (N/p))^\delta} \ll \frac{N}{(\log N)^\delta}\log\log N.
\end{equation}

We split the range of summation in the sum in~\eqref{eq: goal}  into the intervals of type:
$$
p\in \cI_v=\left[\frac{N}{e^{v+1}}, \frac{N}{e^v}\right),
$$ 
where 
$$
\log z \le 
v < \log (N/z).
$$

Now, applying the {\it Mertens formula}, see, for example,~\cite[Equation~(2.15)]{IwKow}, 
$$
\sum_{p\le x}\frac{1}{p} =\log \log x +A+ O\((\log x)^{-1} \)
$$ 
for some constant $A$, we derive
$$
\sum_{z<p\le N/z} \frac{N/p}{(\log (N/p))^\delta}
\ll 
N\sum_{\log z\le v < \log(N/z)}v^{-\delta} \sum_{p\in \cI_v} \frac{1}{p}. 
$$
The inner sum can be estimated as 
\begin{align*}
\sum_{p\in \cI_v} \frac{1}{p}& = 
\log\log (N/e^v) +A +O\left(\frac{1}{\log (N/e^v)} \right)\\
& \qquad -\left(\log\log (N/e^{v+1}) +A+O\left(\frac{1}{\log (N/e^{v+1})} \right)\right)\\
& =\log \left( \frac{\log N(1-v/\log N)}{\log N(1-(v+1)/\log N)}\right)+O\left(\frac{1}{\log (N/e^{v+1})} 
\right) \\
&\ll \frac{1}{\log N-(v+1)}+
\frac{1}{\log (N/e^{v+1})} \ll \frac{1}{\log N-(v+1)}. 
\end{align*}

Therefore we have
\begin{align*}
\sum_{z<p\le N/z} \frac{N/p}{(\log (N/p))^\delta}
& \ll
N\sum_{\log z\le v\le \log(N/z)}v^{-\delta} \sum_{p\in \cI_v} \frac{1}{p}\\
& \ll N \sum_{\log z\le v\le \log(N/z)} 
\frac{v^{-\delta}}{\log N-v}. 
\end{align*}
We split the last sum into two sums depending on the size of $v$
and obtain
\begin{equation} \label{eq: S1S2}
\sum_{z<p\le N/z} \frac{N/p}{(\log (N/p))^\delta} \ll N(S_1 + S_2) , 
\end{equation}
where 
\begin{align*}
S_1 &= \sum_{\log z \le v\le 0.5 \log N} 
\frac{v^{-\delta}}{\log N-v}, \\
S_2 & = \sum_{0.5 \log N \le v\le \log(N/z)} \frac{v^{-\delta}}{\log N-v}. 
\end{align*}
We now estimate $S_1$ and $S_2$ separately.

For $S_1$ we have,
\begin{equation} \label{eq: Bound S1}
\begin{split}
S_1 & \ll \frac{1}{\log N} \sum_{\log z \le v\le 0.5 \log N} v^{-\delta} \le \frac{1}{\log N} \sum_{1\le v\le 0.5 \log N} v^{-\delta}\\
& \ll \frac{1}{\log N} (\log N)^{1-\delta} = (\log N)^{-\delta}. 
\end{split} 
\end{equation}

Furthermore, for $S_2$ we obtain
\begin{equation} \label{eq: Bound S2}
\begin{split}
S_2 & \ll \sum_{0.5 \log N \le v\le \log(N/z)} \frac{v^{-\delta}}{\log N-v} \\
& \ll 
(\log N)^{-\delta} \sum_{0.5 \log N \le v\le \log(N/z)} \frac{1}{\log N-v} \\
& \ll (\log N)^{-\delta} \sum_{\log z -1 \le u \le 0.5 \log N+1} \frac{1}{u} \ll (\log N)^{-\delta} \log \log N.
\end{split} 
\end{equation}
Substituting~\eqref{eq: Bound S1} and~\eqref{eq: Bound S2} in~\eqref{eq: S1S2} we obtain~\eqref{eq: goal},  which after substitution in~\eqref{eq: 3_8-free} yields
\begin{equation} \label{eq: Bound F1}
E \ll N (\log N)^{-\delta} \log \log N.
\end{equation}

Substituting the bound~\eqref{eq: Bound F1}  in~\eqref{eq: N2 Asymp}, we  complete  the proof.\end{proof} 

\begin{rem} We note that when $\delta$ grows and approaching $1/4$, the range  of $z$ in 
Lemma~\ref{12345} is getting broader while the error term is improving. However, the implied
constant in the error term depends on $\delta$ and tends to infinity  when $1/4 -\delta$
tends to zero.
\end{rem}

We note (although we do not need this for our argument) that  for  $z<N^{ 1/(10 \log\log N)}$
we have 
\begin{equation} \label{eq: N asymp}
 \#\cN(N,z) = e^{-\gamma} \frac{N}{\log z} +O\(N/(\log z)^2\),
\end{equation}
where $\gamma = 0.57721\ldots$ is the {\it Euler--Mascheroni constant\/}.
Indeed
 by~\cite[Part~I, Theorem~4.3]{Ten15},
 the set $\cR(N,z)$ of positive integers 
$n\le N$ with all prime factors $p>z$ is of cardinality 
$$
\# \cR(N,z) = N\prod_{p\le z}\left(1-\frac{1}{p}\right)\left(1+O\(\(\log  z\)^{-2} \) \right),
$$ 
provided $z<N^{ 1/(10 \log\log N)}$.
The non square-free numbers contribute at most 
$$
Q \ll \sum_{\substack{n=p^2 m \le N, \\ p>z}}1 \ll \sum_{p>z} \frac{N}{p^2}  
\ll N/z
$$ 
to $R$. 
Recalling the {\it Mertens formula\/}, see ~\cite[Part~I, Theorem~1.12]{Ten15}
$$
\prod_{p\le z}\left(1-\frac{1}{p}\right) = e^{-\gamma} \frac{1}{\log z} +O\(1/(\log z)^2\)
$$
and using  that $\#\cN(N,z)=R - Q$ we derive~\eqref{eq: N asymp}.

\subsection{Smallest non-residues of characters} 

We recall the notation $N_n$ from 
Section~\ref{sec:not}. 
The following results is a special case of~\cite[Theorem~1]{ba2}.
\begin{lemma} \label{ba} 
For arbitrary $\gamma> 2$ and $\varepsilon > 0$ we have 
$$
N_n \leq (\log x)^\gamma
$$ 
for all but $O\left(x^{1/(\gamma-1-\varepsilon)}\right)$
odd, square-free, positive integers $n\le x$.
\end{lemma}

We remark that the next result given by~\cite[Theorem~1.1]{LaWu} applies to more general settings of 
Kronecker symbols. We only need its part for square-free integers. Note that compared to Lemma~\ref{ba} 
it gives a stronger bound on $N_n$ but a weaker bound on the exceptional set. 

\begin{lemma}\label{LW}
There exists an absolute constant $C>0$ such that 
$$
N_n \ll \log n
$$ 
for all but $O(x^{1-C/\log\log x})$ value of 
$n\in \cN(x)$. 
\end{lemma}

\subsection{Modular reductions of elliptic curves} 
\label{secmod red} 

We recall that an integer $a$, relatively prime to a prime $p$, is called a {\it quadratic non-residue modulo $n$\/} if the congruence 
$x^2 \equiv a \mod n$ has no solution in integers. 
The following result is due to Schoof~\cite[Lemma~4.8]{Schoof}. 

\begin{lemma}\label{order}
Let $E_b:~y^2=x(x^2 - b)$ be an elliptic curve over $\F_p$. If $p=3 \mod 4$, then $E_b(\F_p)$ 
has order $p+1$ and is cyclic or is isomorphic to  $\ZZ_2\times \ZZ_{(p+1)/2}$ according to whether $b$ is 
a quadratic non-residue or residue modulo $p$, respectively. 
\end{lemma} 


Note that the condition that $E_b:~y^2=x(x^2 - b)$ is an elliptic curve automatically 
excludes the value $b = 0$.

The following result is given in~\cite[Theorem~4.2]{kn}. 
\begin{lemma} \label{kn} 
Let $E$ be an elliptic curve over the field $\F_q$ of characteristic $\neq 2,3$, 
given by the equation 
$$
y^2 =(x-a)(x-b)(x-c) 
$$
with $a,b, c$ in $\F_q$. For $(x_2,y_2)$ in $E(\F_q)$, there exists $(x_1, y_1)$ in 
$E(\F_q)$ with $2(x_1,y_1)= (x_2,y_2)$ if and only if $x_2-a, x_2-b$ and $x_2-c$ are squares 
in $\F_q$. 
\end{lemma}
Let $|P|$ denote the order of a finite point $P\in E(\Z_n)$. 
We also use $\nu_2(k)$ to denote the $2$-adic order of an integer $k \ge 1$. 

Before we formulate other results, let us make some preparations. 
If $e$ is a known, even exponent of a finite point $P\in E(\Z_n)$ such that the $2$-adic 
orders of $P_i$ 
and $P_j$ are distinct, that is, 
\begin{equation}\label{ord2} \nu_2\(|P_i|\)\neq \nu_2\(|P_j|\) \text{ for some } i \neq j, \end{equation} 
then we can use the above formulas~\eqref{addel} to find a nontrivial divisor of $n$.
To see that let us assume that 
$P\in E(\Z_n)$ is finite (otherwise $\gcd(z_P,n)$ is a divisor of $n$, where $P=(x_P:y_P:z_P)$). 

We recall that the classical  {\it double-and-add algorithm\/}, that is, iterating the formulas
$$
eP = \begin{cases} 2\((e/2) P\), &\quad  \text{if $e$ is even},\\
  (e-1)P + P, & \quad \text{if $e$ is odd},
 \end{cases}
$$
one can compute the multiple  $eP$ of a point $P\in E(\Z_n)$
in $O(\log n)$ arithmetic operations in $\Z_n$.  

Similarly as in the Lenstra factorisation method~\cite{Len1}, computing $eP$ using the above double-and-add algorithm, we find a nontrivial divisor by computing the greatest common divisor of $n$ and the denominator of $\lambda$ above,
when formulas~\eqref{addel} fail for the first time.

Assume that $Q$ is a finite point and an input to a step which fails (denominator of $\lambda$ is not a unit), then we find a non-trivial factor of $n$. More precisely, if $\vp(Q)=(Q_1, \ldots, Q_s)$ and in 
the doubling step formulas~\eqref{addel} fail, then $2Q_i=O$ for some $i$, and it follows from~\eqref{ord2}
that $2Q_j\neq O$ for some $j$, thus $\gcd(n,y_Q)$ is a non-trivial divisor of $n$. 

If $Q$ is an input to the addition step,
then $Q= 2lP$ for some integer $l$ as an output of doubling. 
If the addition step $Q+P$ fails, then $2lP_i = \pm P_i$ for some $i$, but again from \eqref{ord2} it follows that 
$2lP_j \neq \pm P_j$, since otherwise $\nu_2(|P_j|)=\nu_2(|P_i|)=0$, which contradicts the assumption. Hence for some $j$ we obtain that $\gcd(n,x_Q- x_P)$ is a non-trivial divisor of $n$.
This shows that if $2$-adic local orders of some point $P$ of even order are distinct, then we can 
factor $n$ in $\(\log n\)^{2+o(1)}$ bit operations, see~\cite{vzGG} for a background on complexity of arithmetic operations.

\begin{lemma} \label{1234}
Let $n$ be a positive, odd, square-free integer, and let $ p,q\equiv 3 \mod 8$ 
be  prime divisors of $n$.  
Assume that  $b$ is a quadratic 
nonresidue modulo $p$ and a quadratic residue modulo $q$. 
Let $y$ be coprime to $n$  
and $x$ be a quadratic nonresidue modulo $p$ such that $x(x^2-b)$ is not divisible by $p$.
  Then for 
$$
\alpha \equiv x(x^2-b)/y^2 \mod n, 
$$ 
 computing the multiples $k_bE(n,b)P$  
with $k_bE(n,b) = n^{O(1)}$ of  a point  $P=(\alpha x,\alpha^2y)$ lying on the curve $E_{b\alpha^2}$ modulo $n$, 
we   recover a nontrivial divisor of $n$ in at most $\(\log n\)^{2+o(1)}$ bit operations. 
\end{lemma} 

\begin{proof}  First, note that any point  $P=(\alpha x,\alpha^2y)$ 
on 
$$E_{b\alpha^2}:~y^2=x(x^2-b\alpha^2)
$$ 
has order $2^ak \mod pq$, where $k$ is odd and $a=0, 1$ or $2$, by the congruence conditions for $p$ and $q$. 
Using assumptions on $b$, we see that 
$E_{b\alpha^2}(\F_p)$ contains a point of order $4$, while $E_{b\alpha^2}(\F_q)$ 
contains no point of order $4$. Now, by Lemmas~\ref{order} and~\ref{kn},    $P$ has 
maximal $2$-adic  order equal to $2^2$
on $E_{b\alpha^2}(\F_p)$, while it has $2$-adic order at most equal to $2^1$ on $E_{b\alpha^2}(\F_q)$. Therefore  
the preliminary  discussion before the formulation of Lemma~\ref{1234}  implies that 
$n$ can be split nontrivially in  $\(\log n\)^{2+o(1)}$  
 bit operations when computing doubling of the point $Q$ above.
This completes the proof. 
\end{proof}

\section{Proofs  of Main Results}

\subsection{Algorithm}
\label{sec:Alg} 
To prove Theorem~\ref{thm:factoring}, 
we now present an algorithm which finds a nontrivial factor of all but at most $O\(N/(\log N)^{\delta} \)$ integers $n\in \cN(N,z)$ , 
where  $z=(\log N)^\beta$ with 
 \begin{equation}\label{eq:beta}
\beta=\frac{\gamma-1}{8(\gamma-2)}
\end{equation} 
and $\delta$ is arbitrary positive constant 
with $\delta<1/4$,  in  deterministic polynomial time with a \MultEll$(N,B,M)$-oracle  given by Definition~\ref{def:MultEll}.

In the following algorithm we try to compute, using   the double-and-add algorithm, the  multiples $k_bE(n,b)P$ with 
$$k_bE(n,b)\ll k_b \prod_{p\mid n} (2p)\le k_b n^{1+o(1)},
$$    
obtained from  a   \MultEll$(N,B,M)$-oracle, of some points $P$ lying on the curve $E_{b\alpha^2}$ modulo $n$, 
see Section~\ref{secmod red}.

\begin{algor}[Factoring with a  \MultEll$(N,B,M)$-oracle]
\label{alg:AlgA} { } \qquad \\ 

\noindent
{\bf Input:} Square-free positive integer $n$, parameters $\gamma> 0$ and 
$B = \fl{(\log N)^\gamma}$
and
a \MultEll$(N,B,M)$-oracle. 

\smallskip
  
\noindent
{\bf Output:} Factorisation of $n$ or answer $n$ is exceptional. 

\begin{enumerate} 
\item Set $z =  (\log N)^\beta$, where $\beta$ is given by~\eqref{eq:beta}. 

\item Search for prime divisors $p\mid n$ with $p\le z$,
using trial division.   If a nontrivial divisor $p$ of $n$ is discovered then stop and return $p$.

\item For $b=1,2 \ldots B $, take $y=1$, $x=b$, $\alpha =x(x^2-b) \mod n$ and $P=(\alpha x, \alpha^2y)$. Next, if the \MultEll$(N,B,M)$-oracle  from Definition~\ref{def:MultEll} returns  
$\star\star\star$ we declare $n$ exceptional and terminate. 
Otherwise we use the product $k_b E(n,b)$  try to compute  $k_b E(n,b)P$ using the double-and-add algorithm.  
If during the computations a nontrivial divisor $d\mid n$ is discovered then stop and return $d$ as output. 
Otherwise output  $n$ as exceptional.   \end{enumerate} 
\end{algor}

Below, for each $N$,  we use Algorithm~\ref{alg:AlgA} with some specific parameters $B$ and $M$.

\subsection{Concluding of the proof of  Theorem~\ref{thm:factoring}}
For the proof we choose 
$$
z=(\log N)^{\frac{\gamma-1}{8(\gamma-2)}}
$$
where 
$$
2 < \gamma <2+\frac{\sqrt {257} -15}{16}
$$
and  fix some sufficiently small $\varepsilon>0$.

Note that the upper bound on $\gamma$ is chosen to guarantee 
 \begin{equation}\label{eq:beta>gamma}
\beta > \gamma
\end{equation} 
where $\beta$ is given by~\eqref{eq:beta}.

Clearly any positive  integer $n \le N$  with a prime divisor $p\le z$ is 
 factored during Step~(2) of Algorithm~\ref{alg:AlgA}. 
 
Clearly,  the  number of non-square-free integers for which  Step~(3)  applies and 
returns $\star\star\star$ is at most 
$$
\sum_{p \ge z} N/p^2  \ll N/z, 
$$
which give an admissible contribution to the exceptional set. 

Therefore, from this point on, we always 
assume $n \in \cN(N,z)$, where  $\cN(N,z)$ is  as in Section~\ref{sec:sp int}.

In view of Lemma~\ref{12345}, applied with say 
$\widetilde \delta = 1/8 +\delta/2$  instead of $\delta$, we see that for any $\delta < 1/4$ (thus $ \delta < \widetilde \delta  <1/4$)
all  numbers $n\in \cN(N,z)$ except at most the set $\cE_1$ of cardinality $\# \cE_1\ll N/(\log N)^\delta$ can be represented in the form $pqm$, where $p\neq q$ are primes such that $ p,q \equiv 3 \mod 8$. 
We  prove that among them there are at most $\#\cE_2<N(\log N)^{-2\gamma (1-\vartheta)}$ exceptions which are  be factored in 
Steps~(2) and~(3)
of Algorithm~\ref{alg:AlgA}, where $\vartheta=1/(\gamma -1) $.
In view of Lemma~\ref{1234} it is sufficient to prove that all  numbers $n\in \cN(N,z) \setminus\( \cE_1  \cup \cE_2\)$ are
  factored in Steps~(2) and~(3).
  
The set of exceptions $\cE_2$ 
consists of those numbers $n=pqm$ such that $ N_{pq}>B$. 
 Namely if $b\le B$ is such that $b$ is a quadratic nonresidue moduli $pq$ then  $b$ is a quadratic nonresidue modulo  exactly one $r\in \{p,q \}$
 one  and a quadratic residue modulo the complementary factor $pq/r$.  Without loosing the generality let us assume that $r=p$. Then 
$$
\alpha=x(x^2-b)=b(b^2-b),
$$
 where $x$ is a quadratic nonresidue modulo $p$ and moreover $b^2-b$ is not divisible by $p$ since $b$ is not a square modulo $p$. 
 Therefore, recalling~\eqref{eq:beta>gamma},  we obtain the 
 inequality $P^{-}(n)> z\ge B$, where $P^{-}(n)$ is  the smallest  prime divisor of $n$.  
 
We  see that the assumptions of Lemma~\ref{1234} are satisfied, since the discriminant  of the curve $E(b)$, that is $\Delta_{E(b)}=-4b^3$,   is coprime to $n$.
To complete the proof, let $S(t)$  be the counting function of 
odd, square-free positive integers $n\le t$ such that  $N_n> B(t)$,  
where 
 $$
 B(t) = \fl{(\log t)^\gamma}, 
 $$  
In particular,   by  Lemma~\ref{ba} we have
\begin{equation} \label{eq: Bound St}
S(t)\ll t^{\vartheta}
\end{equation}
where for any $\varepsilon > 0$ we can take $\vartheta=1/(\gamma-1-\varepsilon)$.  

Since $pq\ge z^2$ we estimate $\cE_2$
as follows:
\begin{align*}
\cE_2
& \ll\int_{z^2}^{N} \frac{N}{t} dS(t) =  S(N) - S(z^2) + N \int_{z^2}^{N} \frac{S(t)}{t^2} dt\\
&    \ll 
N \int_{z^2}^{N} t^{-2+\vartheta}dt\ll  N (z^2)^{-1+\vartheta }
 \ll N \(\log N\)^{-\delta},
 \end{align*}  
provided   $\varepsilon > 0$ is small enough, 
since  for $\gamma >2$ we have
$$ 
2\frac{\gamma-1}{8(\gamma-2)}\( 1-\frac{1}{\gamma-1}\)= 1/4>\delta.
$$  
In view of the preliminary discussion before the formulation of Lemma \ref{1234} we see that Step~(2) of the above algorithm takes   $\ll z \log ^{1+o(1)} N \ll (\log N)^{\beta+1+o(1) } =(\log N)^{1+\frac{\gamma-1}{8(\gamma-2)}+o(1)}$ bit operations, while Step~(3) takes $B(\log N)^{2+o(1)}=(\log N)^{\gamma+2+o(1)}$.
This  implies that  the complexity of  Algorithm~\ref{alg:AlgA}  is  $(\log N)^{\rho +o(1)}$, where
$$
\rho = \max\left \{\gamma+2, \beta+1 \right\}
= \max\left \{\gamma+2, \frac{9\gamma-17}{8(\gamma-2)} \right\}
$$  
with the  related number of exceptions   of order $O\(N/(\log N)^{\delta}\)$ for  any positive  $\delta<1/4$.

\subsection{Concluding of the proof of Theorem~\ref{thm:factoring2}}
\label{sec:concl} 

We apply a slight modification in Algorithm~\ref{alg:AlgA} 
where the parameters $z$ and $B$ are now chosen as
\begin{equation} \label{eq: z B}
z=B=\fl{c\log N}
\end{equation}
 where $c$ is the constant implied by the symbol $\ll$ in Lemma~\ref{LW}, 
and also 
$k_b$ is  replaced by $k, M=N^{O(B)}$,  and \MultEll$(N,B,M)$-oracle is replaced by
 \tMultEll$(N,B,M)$-oracle given by Definition~\ref{def:MultEll-t}.

Hence proceeding    as above 
 we estimate $\cE_2$ by Lemma~\ref{LW}, we have $S(t)\ll t^{1-C/\log\log t}$, 
 instead of~\eqref{eq: Bound St},  
 thus obtaining for 
 $$
 B(t) = \fl{c \log t }, 
 $$  
 where $c$ is chosen as in~\eqref{eq: z B}, 
 the following bound
\begin{align*} 
\cE_2 & \ll \sum_{1\le m\le N/z^2}\ \sum_{\substack{z^2< s< N/m\\ N_s>B(N/m)\\s~\text{square-free}}} 1 \\
& \ll \sum_{1\le m\le N/z^2}  \( \frac{N}{m}\)^{1-C_0/\log\log (N/m)}\(\log\frac{N}{m}\)^{-2}
\end{align*} 
for some $0<C_0<C$, since $u^{C/\log\log u}>\(\log u\) ^2 $ for sufficiently large $u> u_0$. Therefore letting 
$\delta(u)=C_0/\log\log u$ we obtain, considering the case $m=1$ and $m\in [z, N/z^2]$ separately,  
that
\begin{align*} 
\cE_2&  \ll N^{1-C_0/\log\log N}+
\int_ z^{N/z^2} \(\frac{N}{t}\)^{1-\delta(\frac{N}{t})} \(\log \frac{N}{t} \)^{-2}dt\\
&  \ll N^{1-C_0/\log\log N}+
N \int_ {z^2}^{N/z} u^{-1-\delta(u)}(\log u)^{-2} du\\
& \ll N^{1-C_0/\log\log N}+
N\max_{v\in [z^2, N/z]} v^{ -C_0/\log\log v}\int_{z }^{N/z^2}\frac{du}{u\log^2 u}\\
& \ll N (\log N)^{-C_0 /\log\log\log N}
\end{align*} 
when letting $z=\fl{c\log N}$.  

The complexity of algorithm follows by remarking that for the input data of order $N^{O(B)}$ the number of addition and doublings is $O\(B\log N \)$ while their  complexity is of order $\(\log N\)^{1+o(1)}$ giving (for $B = \fl{c \log N }$) 
altogether the bound $\(\log N\)^{3+o(1)}$ on the algorithm complexity bound, as required. 
The oracle is queried only once for  a multiple $kF(n,B) \le M$,
 where $M=N^{O(B)}$. 
 
 \section{Comments}
 
We remark that under the GRH, by the classical result of Ankeny~\cite{Ank}
for any $\gamma>2$ the exceptional sets of Lemma~\ref{ba},  and thus of  Theorem~\ref{thm:factoring}, is finite. 
We note that Ankeny~\cite{Ank} is interested in the smallest quadratic non-residue $b$ modulo $n$ 
(that is, the smallest $b$ for which $x^2 \equiv b \mod n$ has no solutions), but~\cite[Theorem~1]{Ank} 
in fact gives a bound on the larger quantity $N_n$. 
 On the other hand, it is not clear whether the 
GRH allows to improve   Lemma~\ref{LW} and Theorem~\ref{thm:factoring2}.

 \section*{Acknowledgement}

 The authors are very grateful to Drew Sutherland for very useful comments
 and to the referees for the very careful reading of the manuscript and helpful 
 suggestions. 
 
This work started during a very enjoyable visit by I.S. to the   Department of Mathematics 
of  the Warsaw University, whose support and hospitality are gratefully acknowledged.

This work was  supported   by ARC Grant~DP170100786.

\end{document}